\documentclass[psamsfonts]{amsart}
\usepackage{amsmath,amstext,amssymb,amsopn,amsthm}
\usepackage{amsmath,amssymb,amsthm}
\usepackage[mathscr]{eucal}
\usepackage{bbm}
\usepackage{color}
\usepackage{tikz}
\usetikzlibrary{arrows}
\usepackage[unicode,bookmarks,colorlinks]{hyperref}
\hypersetup{
    linkcolor=brickred,
}
\usepackage{multirow}
\usepackage{bigstrut}

\definecolor{mahogany}{cmyk}{0, 0.77, 0.87, 0}
\definecolor{salmon}{cmyk}{0, 0.53, 0.38, 0}
\definecolor{melon}{cmyk}{0, 0.46, 0.50, 0}
\definecolor{yellowgreen}{cmyk}{0.44, 0, 0.74, 0}
\definecolor{brickred}{cmyk}{0, 0.89, 0.94, 0.28}
\definecolor{OliveGreen}{cmyk}{0.64, 0, 0.95, 0.40}
\definecolor{RawSienna}{cmyk}{0, 0.72, 1.0, 0.45}
\definecolor{ZurichRed}{rgb}{1, 0, 0} 

\numberwithin{equation}{section}

\newtheorem{thm}{Theorem}[section]        
\newtheorem{cor}{Corollary}[section]
\newtheorem{lem}{Lemma}[section]
\newtheorem{prop}{Proposition}[section]

\newtheorem{definition}{Definition}[section]

\newenvironment{•}{•}{•}
\newcommand{\R}{\mathbb{R}}                  
\newcommand{\Rd}{\R^d}               
\newcommand{\iord}{\int_{\Rd}}              
\newcommand{\intd}[1]{\int_{#1} }    
\newcommand{\set}[1]{ \left\{#1\right\} }
\newcommand{\mysum}[3]{\sum\limits_{#1=#2}^{#3}}          

\newcommand{\abs}[1]{\left|#1\right|}

\newcommand{\tgo}{t \rightarrow 0+}

\newcommand{\Prob} {\mathbb{P} }
\newcommand{\palp}{p_{t}}
\newcommand{\Halp}{\mathbb{H}_{B}(t)}

\newcommand{\pthesis}[1]{\left(#1\right)}

\newcommand{\dom}{B}
\newcommand{\cvf}{g_{\dom}}
\newcommand{\indf}[1]{\mathbbm{1}_{#1}}

\newcommand{\lo}{2}

\newcommand{\bb}{\mathbb{S}^{d-1}}
\newcommand{\ra}{\rightarrow}
\usepackage[margin=1.25in]{geometry}
\DeclareMathOperator{\arcsinh}{arcsinh}
\begin{document}

\title[ Heat content for Poisson kernel]{On the heat content for the Poisson kernel over the unit ball in the euclidean space.}
\author{Luis Acu\~na Valverde}
\address{Department of Mathematics, Universidad de Costa Rica, San Jos\'e, Costa Rica.}
\email{luis.acunavalverde@ucr.ac.cr/ guillemp22@yahoo.com}
\maketitle

\begin{abstract}
This paper studies, by employing analytical tools, the small time behavior of the heat content for the Poisson kernel over  the unit ball in $\Rd$, $d\geq 2$ by working with its  related set covariance function. As a result, we obtain a representation for the third term in the expansion of the heat content over the unit ball and provide the explicit form of the this term in the particular cases $d=2$ and  $d=3$.
\end{abstract}

\section{introduction}
Let $d\geq 2$ be an integer. We consider the Poisson heat kernel defined by
\begin{align}\label{cd}
p_t(x)=\frac{k_d\,t}{(t^2+|x|^2)^{\frac{d+1}{2}}},\,\,x\in \R^d, \,\, t>0,\,
\end{align}
where  $|x|$  shall stand for the usual norm of $x\in \R^d$  and
\begin{align}\label{kappadef}
k_d=\frac{\Gamma\pthesis{\frac{d+1}{2}}}{\pi^{\frac{d+1}{2}}}.
\end{align}

We observe that the Poisson heat kernel has the following properties.
\begin{enumerate}
\item[$i)$] For all $t>0$, $p_t(x)$ is a radial function. That is,   $p_t(x)=p_t(|x|\,e_d)$ for all $t>0$ and $x\in \Rd$ where  $e_d=(0,...,0,1)\in \R^d$ . Moreover, for all $t>0$, we have
\begin{align}\label{intvalue1}
\int_{\R^d}dx\,p_t(x)=1.
\end{align}
$ii)$ The heat kernel $p_t(x)$ satisfies the scaling property:
\begin{align*}
p_t(x)=t^{-d}p_1(t^{-1}x), t>0, x\in \Rd.
\end{align*}
\end{enumerate}

Before continuing, we provide some useful notations. Throughout the paper,   for $r>0$ and $x\in\Rd$, $B_r(x)$ will denote the ball in $\Rd$ centered at $x$ with radius $r$, where for simplicity the unit ball $B_1(0)$ will be denoted by $B$. Moreover, $\bb$ will denote the boundary of the unit ball $B$. Henceforth, the volume  and   surface area of the unit ball $B$ in $\Rd$ will be denoted by
$|B|$ and $|\bb|$, respectively. It is also convenient to denote the volume and surface area of the unit ball in $\R^k$ for $k\leq d$ by $w_k$ and $A_k$, respectively. That is, 
\begin{align}\label{vaunitball}
w_k=\frac{\,\pi^{\frac{k}{2}}}{\Gamma\pthesis{1+\frac{k}{2}}} \,\,\mbox{and}\,\,A_k=kw_k=\frac{2\pi^{\frac{k}{2}}}{\Gamma\pthesis{\frac{k}{2}}}.
\end{align}

With the appropriate geometric objects already introduced, we proceed to present the function to be investigated. 

\begin{definition} We  define
\begin{align}\label{defhc}
\Halp=\intd{\dom}dx\intd{\dom}dy\,\palp(x-y),
\end{align} 
which will be called {\it the heat content}  over the unit ball  in $\Rd$ with respect to the Poisson heat kernel. 
\end{definition}

In order to motivate our work and show the connections of our results to other areas as probability and spectral theory, we point out that there exists  a L\'evy stochastic process called the Cauchy stable process, denoted by $X=\set{X_t}_{t\geq0}$ having transition densities  given by
identity \eqref{cd}. That is, for any $\Omega\subset\Rd$ Lebesgue measurable set, we have
\begin{align}\label{probc}
\Prob\pthesis{X_t\in \Omega|X_0=x}=\int_{\Omega}p_t(x-y)dy,
\end{align}
which implies that  $ \Halp$ defined in \eqref{defhc} can be rewritten in probabilistic terms by 
$$ \Halp=\int_{B}dx\,\Prob\pthesis{X_t\in B|X_0=x}.$$ This last identity tells us that  $\Halp$ is bounded over $ [0,\infty)$, since 
$$0\leq \Halp\leq \intd{\dom}dx = |B|<\infty,$$
where we have used that any probability measure is  non-negative  and upper bounded by one.

The interest in studying the heat content over the unit ball arises from the paper \cite{ Acu1, Acu},  where estimates concerning   the   behavior of $\Halp$ as $\tgo$ provide  an insight of the possible small time behavior of {\it the spectral heat content for the unit ball} which is defined by 
\begin{align}\label{shc}
Q_{B}(t)=\int_{B}dx\,\int_{B}dy\,\,p^{B}_t(x,y),
\end{align}
where $p^{B}_t(x,y)$ is the transition density for the  Cauchy stable process killed upon exiting $B$. Namely,
\begin{equation}\label{tran.den.dom}
p^{B}_t(x,y)=\palp(x,y)\,\,\mathbb{P}\left(X_s\in B ,\,\,\forall \, 0\leq s\leq t \,\,|\,\, X_0=x,\,\,X_t=y\right).
\end{equation}

The name {\it spectral heat content}  given to $Q_{B}(t)$ comes  from the fact that $p^{B}_t(x,y)$ can be written in terms of the eigenvalues and the eigenfunctions associated to the unit ball. That is,  it is known (see \cite{Davies} for details) that there exists an orthonormal basis of eigenfunctions $\{\phi_n \}_{n\in \mathbb{N}}$ for $L^2(B)$  with  corresponding eigenvalues $\{\lambda_n \}_{n \in \mathbb{N}}$
satisfying $0 < \lambda_1 < \lambda_2 \leq \lambda_3 \leq . . .$
and  $\lambda_n \rightarrow \infty$ as $n \rightarrow \infty$ such that
\begin{equation}\label{spect.dec.palp}
p^{B}_t(x,y)=\mysum{n}{1}{\infty}e^{-t\lambda_n}\,\phi_{n}(x)\,\phi_n(y).
\end{equation}

Notice that due to \eqref{shc} and the last equality, we obtain an expression for 
$Q_{B}(t)$ involving both the spectrum $\set{ \lambda_n}_{ n\in \mathbb{N}}$ and  the eigenfunctions $\set{\phi_{n}}_{n\in \mathbb{N}}$. Namely,
$$Q_{B}(t)=\mysum{n}{1}{\infty}e^{-t\lambda_{n}}\pthesis{\int_{B}dx\,\phi_n(x)}^2.$$


We observe, by appealing to $\eqref{tran.den.dom}$, that $0\leq p_t^{B}(x,y)\leq p_t(x,y)$  which in turn implies that $0\leq Q_{B}(t)\leq \Halp$ for all $t>0$. The last inequality combined  with some other facts exposed in \cite{ Acu1, Acu, Acu4} point out   that the small time behaviors of $Q_{B}(t)$ and $\Halp$ might be similar. It is worth-mentioning that the study of the spectral heat content and heat content started by considering first the Gaussian kernel whose expansions contain geometric information  over the underlying sets where these functions are defined. We refer the interested to  \cite{vanden3, van8, van6, Mir} for further details in these topics and \cite{Cyg, TPS} for results concerning the heat content and spectral heat content associated to other L\'evy processes besides of the Brownian motion.

In  \cite{Acu}, the following conjecture is given.\\
{\bf Conjecture:} For $\Omega\subset \R^d$ an open bounded set with smooth boundary,  the following limit 
\begin{align}\label{lim}
\lim\limits_{\tgo}\frac{1}{t}\pthesis{|\Omega|-\mathbb{H}_{\Omega}(t)-\frac{Per(\Omega)}{\pi}t\ln\pthesis{\frac{1}{t}}}
\end{align} 
exists, where $|\Omega|$ and $Per(\Omega)$ stand for the volume and surface area of the boundary of $\Omega$, respectively. {\it The purpose of this paper is to show that the conjecture holds true for the unit ball in $\Rd$.}  

The main idea in the proof  is based on the fact that the heat content $\Halp$ can be expressed in terms of the set covariance function related to $B$, which we proceed to introduce. 


\begin{definition}[\bf Set covariance function]
 The covariance function of $\dom$ is denoted by $\cvf(z)$ and defined for each $z\in \Rd$ by
\begin{align}\label{gdef}
\cvf(z)=\abs{\dom \cap \pthesis{\dom+z}}=\iord dx\,\indf{\dom}(x)\,\indf{\dom}(x-z).
\end{align}
\end{definition}
Thus, $\cvf(z)$ provides the volume of that domain obtained by intersecting the balls $B$ and $B+z=B_1(z)$ with radii one.

In order to state our result, we need the following function defined for $s>0$,
\begin{align}\label{impfunct}
\gamma_{\dom}(s)=w_{d-1}-\frac{\pthesis{\cvf\pthesis{0}-\cvf\pthesis{se_d}}}{s},
\end{align}
where we recall that $w_{d-1}$ denotes the volume of the unit ball in $\R^{d-1}$ and $e_d=(0,...,0,1)\in \Rd$.

With all the required details being introduced, we continue to  present our main result. 

\begin{thm}\label{mthm}  Consider $\dom$ the unit ball in $\Rd$, $d\geq 2$.  Let $\kappa_d$ and $\gamma_{B} $  be  defined as in \eqref{kappadef}, \eqref{impfunct}, respectively.  Then,
\begin{align}\label{implim}
\lim\limits_{\tgo}&\frac{1}{t}\pthesis{|\dom|-\Halp-\frac{|\bb|}{\pi} t\ln\pthesis{\frac{1}{t}}}=\\  \nonumber 
&|\bb|\pthesis{\kappa_d\set{\frac{|B|}{2}-\int_{0}^{1}ds\,s^{-1}\gamma_{\dom}(\lo\,s)}+\frac{1}{\pi}\set{2\ln(2)+\int_{0}^{\infty}d\theta \pthesis{\tanh^d(\theta)-1}}}.
\end{align}
\end{thm}

The paper is organized as follows. In \S \ref{sec:prelim}, we provide some preliminary results concerning the set covariance funtion $g_B(z)$ needed to show that the term
$$\int_{0}^{1}ds
s^{-1}\gamma_{\dom}(\lo\,s)$$
is finite. In \S \ref{sec:mthproof}, we provide the proof of Theorem \ref{mthm} with the aid of a series of propositions. Finally,  we explicitly compute the third term in Theorem \ref{mthm} for $d=1$ and $d=2$.


\section{properties of the set covariance function of the unit ball.}\label{sec:prelim}
 In  order to compute the limit given in Theorem \ref{mthm} for the unit ball, we require to find the explicit expression for $g_B(z)$ which can be done because of the symmetry of the ball. We mentioned before that $g_B(z)=|B\cap(B+z)|$
represents the volume of the intersection of two balls of radii one, so that it is geometrically clear  that  
\begin{align}\label{scvf}
supp(g_{B})=\overline{\set{z\in \R^d: g_B(z)\neq 0}}=B_2(0).
\end{align}
The following lemma tells us that $g_B(z)$ is a radial function.
\begin{lem} For all $z\in \Rd$, we have that
$g_B(z)=g_B(|z|e_d)$
where $e_d=(0,...,0,1)\in \Rd$.
\end{lem}
\begin{proof}
 We only need to show that 
 $g_B(Tz)=g_B(z)$ for all $T:\Rd\ra \Rd$ orthogonal linear transformation. That is, $T$ satisfies that $\abs{Tz}=\abs{z}$
 for all $z\in \Rd$. It is also known that $T^{-1}$ is an orthogonal linear transformation and $\abs{det(T)}=1$. 
 
 Using the change of variables $x=Tu$ and the aforementioned properties about orthogonal linear transformations, we arrive at
 \begin{align}\label{cvft}
 g_B(Tz)&=\iord dx\,\indf{B}(x)\,\indf{B}(x-Tz)=\iord du \abs{det(T)}\,\indf{B}(Tu)\,\indf{B}(Tu-Tz)\\ \nonumber
 &=\iord du \,\indf{B}(Tu)\,\indf{B}(T(u-z)).
 \end{align}
 
By appealing to the fact that $T(B)=B=T^{-1}(B)$, it follows that $\indf{B}(Tu)=\indf{B}(u)$ and $\indf{B}(T(u-z))=\indf{B}(u-z)$ which together with \eqref{cvft} imply $g_B(Tz)=g_B(z)$.
 \end{proof}
 
 The following lemma provides a formula for the volume of  the intersection of two balls  in $\Rd$ with radii one whose proof can be found in \cite{Acu1}. The lemma also gives the explicit expression for $g_B(ae_d)$, $ 0\leq a\leq 2$ for $d=2$ and $d=3$  whose proof is omitted since they are obtained by appealing to elementary integration tools.

\begin{lem}\label{impl} 
Let $d\geq 2$ be an integer and  $0\leq a\leq 2$.  Then,  we have
\begin{align}\label{volball}
g_B(a\,e_d)=\abs{B\cap B_1(ae_d)}=2\,A_{d-1}\,\Theta\pthesis{\sqrt{1-\frac{a^2}{4}}}-a\,w_{d-1}\,\pthesis{1-\frac{a^2}{4}}^{\frac{d-1}{2}},
\end{align}
where 
$$\Theta(\xi)=\int_{0}^{\arcsin(\xi)}d\theta\,\sin^{d-2}(\theta)\,\cos^{2}(\theta),$$
for $0\leq \xi\leq 1$ and $w_{d-1},A_{d-1}$ as defined in \eqref{vaunitball}. 
In particular, we obtain
\begin{enumerate}
\item[$i)$] \begin{align*}
\Theta(1)&=\frac{|B|}{2A_{d-1}},\\  \nonumber
\Theta(z)&=\begin{cases}
    \frac{1}{2}\pthesis{\arcsin(\xi)+\xi\sqrt{1-\xi^2}}       & \quad \text{if } d=2,\\
   \frac{1}{3}\pthesis{1-(1-\xi^2)^{\frac{3}{2}}}  & \quad \text{if } d =3.\\
  \end{cases}
\end{align*}
\item[$ii)$] 
\begin{align*}
g_B(a\,e_d)=\begin{cases}
    2\arcsin\pthesis{\sqrt{1-\frac{a^2}{4}}}-a\sqrt{1-\frac{a^2}{4}},     & \quad \text{if } d=2,\\
  \frac{4\pi}{3}\pthesis{1-\frac{a^3}{8}}-\pi a\pthesis{1-\frac{a^2}{4}},  & \quad \text{if } d =3.\\
  \end{cases}
\end{align*}
\end{enumerate}
\end{lem}

With the previous estimates at hand, we proceed to show the finiteness of $\int_{0}^{1}dss^{-1}\gamma_{B}(2s)$ for all  $d\geq 2$ integer.

\begin{prop}\label{prop1} Consider $\gamma_B$ as defined in \eqref{impfunct}. Then, for all $d\geq 2$ integer, we have for $0<s\leq 1$ that
\begin{align}\label{ineq00}
\gamma_{B}(2s)&=w_{d-1}\set{1-(1-s^2)^{\frac{d-1}{2}}}-A_{d-1}\set{\frac{\Theta(1)-\Theta(\sqrt{1-s^2)}}{s}},
\end{align}
where $w_{d-1}$, $A_{d-1}$ and $\Theta$ as defined in \eqref{vaunitball} and Lemma \ref{impl}, respectively. As a result,
\begin{enumerate}
\item[$i)$] $\gamma_{B}(2s)\geq 0$ for all $0<s\leq 1$ and
\item[$ii)$]\begin{align}\label{ineq000}
0\leq \int_{0}^{1}dss^{-1}\gamma_{B}(2s)&\leq \frac{w_{d-1}\sigma_{d}}{2},
\end{align}
 with $\sigma_{d}=\mathbbm{1}_{\set{2}}(d)+\frac{d-1}{2}\cdot \mathbbm{1}_{[3,\infty)}(d)$.
 \end{enumerate}
\end{prop}

\begin{proof}
 We observe by  Lemma \ref{impl} that $g_{B}(0)=|B|=2A_{d-1}\Theta(1)$. Thus, by appealing to  the identity \eqref{volball}, we obtain that
\begin{align}
\frac{g_B(0)-g_B(2se_d)}{2s}
 =A_{d-1}\pthesis{\frac{\Theta(1)-\Theta\pthesis{\sqrt{1-s^2}}}{s}}+w_{d-1}(1-s^2)^{\frac{d-1}{2}},
\end{align}
for $0< s\leq1$. Therefore,  the desired identity \eqref{ineq00} is an easy consequence of the above equality together with \eqref{impfunct}.

To show $i)$, we need to use that the funci\'on $\Theta$ defined in Lemma \ref{impl}  can be rewritten as
\begin{align}\label{otheta}
\Theta(\xi)=\int_{0}^{\xi}dr\, r^{d-2}\,\sqrt{1-r^2}, 0\leq \xi\leq 1
\end{align} after a suitable change of variables, which in turn yields that
\begin{align}\label{ineqbasic}
\Theta(1)-\Theta(\sqrt{1-s^2})&=\int_{\sqrt{1-s^2}}^{1}dr\, r^{d-2}\,\sqrt{1-r^2}\\ \nonumber&\leq s\int_{\sqrt{1-s^2}}^{1}dr\, r^{d-2}=\frac{1}{d-1}s\pthesis{1-(1-s^2)^{\frac{d-1}{2}}}.
\end{align}
where we have used that $0\leq \sqrt{1-r^2}\leq s$
provided that $\sqrt{1-s^2}\leq r\leq 1$.
Therefore, $i)$ follows from the  inequality \eqref{ineqbasic} since by $\eqref{vaunitball}$, we have that $$\frac {w_{d-1}}{A_{d-1}}=\frac{1}{d-1}.$$

On the other hand, by appealing to the Bernoulli's inequality  $1-\sigma_{d}x\leq (1-x)^{\frac{d-1}{2}}$ for $0\leq x \leq 1,$  we conclude that
$1-(1-s^2)^{\frac{d-1}{2}}\leq \sigma_{d}s ^2$ for $0<s\leq 1$. Hence, due to \eqref{ineq00} and the fact that $\Theta(\xi)$ is a non-decreasing function (see \eqref{otheta}) as long as $\xi\in [0,1]$, we arrive at $$0\leq s^{-1}\gamma_{\dom}(2s)\leq  w_{d-1}s^{-1}\set{1-(1-s^2)^{\frac{d-1}{2}}}\leq  w_{d-1}\sigma_{d}s,$$ for all $0<s\leq 1$. Hence, inequality \eqref{ineq000} is obtained by  integration and this finishes the proof.
\end{proof}

\section{proof of theorem \ref{mthm}}\label{sec:mthproof}

The key step to proving Theorem \ref{mthm} consists on expressing the heat content $\Halp$ in terms of the set covariance function $g_{\dom}(z)$ whose support is $B_2(0)$  and appeal to spherical coordinates   to obtain a  suitable decomposition.

To begin with, by applying Fubini's Theorem and performing a simple change of variable, we have  based on  \eqref{defhc} and \eqref{gdef} that
\begin{align*}
\Halp=\iord dx\iord dy \,\palp(x-y)\indf{\dom}(y)\indf{\dom}(x)=\iord dz\,
\palp(z)\,\cvf(z).
\end{align*}
By using the scaling property of the Poisson heat kernel, namely, $p_t(x)=t^{-d}p_1(t^{-1}x)$ and the change of variable $w=t^{-1}\,z$, we arrive at
\begin{align*}
\Halp=\iord dz\,
\palp(z)\,\cvf(z)=\iord dz\,t^{-d}
p_1(t^{-1}z)\,\cvf(z)=\iord dw\,
p_1(w)\cvf\pthesis{tw}.
\end{align*}


Next, the facts that $\cvf(tw)=0$ only when $|tw|\geq 2$ (see \eqref{scvf}) and  $\cvf(0)=|\dom|$ lead to the following decomposition of $\Halp$.
\begin{align}\label{ee1}
\Halp&=\int_{|tw|< 2} dw\,
p_1(w)\cvf\pthesis{tw}\\ \nonumber&=|\dom|\int_{|tw|< \lo}dw\,p_1(w)+
\int_{|tw|<\lo}dw\,p_1(w)\pthesis{\cvf(tw)-\cvf(0)}\\ \nonumber
&=|\dom|-|\dom|\int_{|tw|\geq \lo}dw\,p_1(w)+
\int_{|tw|<\lo}dw\,p_1(w)\pthesis{\cvf(tw)-\cvf(0)},
\end{align}
where we have also used, according to \eqref{intvalue1},  that 
$$1=\int_{\Rd}dw\,p_1(w)=\int_{\Omega}dw\,p_1(w)+\int_{\Omega^c}dw\,p_1(w)$$
for any domain $\Omega\subset\Rd$.
Therefore, we conclude from  the identity \eqref{ee1} that 
\begin{align}\label{id1}
|\dom|-\Halp= |\dom|\phi(t)+ I_{\dom}(t),
\end{align}
where we have defined 
\begin{align}\label{Mdef}
\phi(t)&=\int_{|tw| \geq \lo}dw\,p_1(w),\\ \nonumber
I_{\dom}(t)&=\int_{|tw|< \lo}dw\,p_1(w)\pthesis{\cvf(0)-\cvf(tw)}.
\end{align}

The following proposition provides further properties concerning the function $\phi(t)$ previously defined.
\begin{prop}\label{prop1 }
Consider $\phi(t)$ for $t\geq 0$ as defined in \eqref{Mdef}. Then,
\begin{enumerate}
\item[$i)$] $0\leq \phi(t)\leq 1 $ is a nondecreasing function on $[0,\infty)$ with  $\lim\limits_{\tgo}\phi(t)=0$.\\
\item[$ii)$]
\begin{align*}
\lim_{\tgo}\frac{\phi(t)}{t}=\frac{|\bb|\kappa_d}{2},
\end{align*}
with  $\kappa_d $ as defined in  \eqref{kappadef}.
\end{enumerate}
\end{prop}

\begin{proof}
To begin with,  by appealing to \eqref{probc}, we have that $\phi(t)=\Prob\pthesis{|X_1|\geq 2t^{-1} |X_0=0}$ which implies that $0\leq \phi(t)\leq 1$. Next, by using  spherical coordinates and the definition of the Poisson heat kernel \eqref{cd}, we have for $t>0$ that
$$\phi(t)=\kappa_d \int_{|tw|\geq 2}\frac{dw}{(1+|w|^2)^{\frac{d+1}{2}}}=|\bb|\,\kappa_d\int_{2\,t^{-1}}^{\infty}\frac{dr\,r^{d-1}}{(1+r^2)^{\frac{d+1}{2}}}.$$
By considering the change of variable $\zeta=r^{-1}$, we obtain that
\begin{align*}
\phi(t)=|\bb|\,\kappa_d\int_{0}^{t/2}\frac{d\zeta}{(1+\zeta^2)^{\frac{d+1}{2}}},
\end{align*}
from which we conclude $i)$. Moreover, it is clear that $\phi(t)$ is differentiable over $(0,\infty)$ with
$$\phi'(t)=\frac{|\bb|\,\kappa_d}{2\pthesis{\frac{t^2}{4}+1}^{\frac{d+1}{2}}},$$
due to the Fundamental Theorem of Calculus.
Then, by part $i)$, we can apply L'Hopital's rule to obtain 
\begin{align*}
\lim_{\tgo}\frac{\phi(t)}{t}=\lim_{\tgo}\phi'(t)=\frac{|\bb|\kappa_d}{2},
\end{align*}
 which proves $ii)$ and this finishes the proof.
\end{proof}

We now proceed to investigate the function $I_{\dom}(t)$ defined in \eqref{Mdef}.
\begin{prop}\label{prop22} Consider  $I_{\dom}(t)$ as defined in \eqref{Mdef} for $t>0$ and set
\begin{align}\label{2impf}
\Psi(t)&=\int_{0}^{\lo\,t^{-1}}\frac{dr\,r^d}{(1+r^2)^{\frac{d+1}{2}}}.
\end{align}

Then,
\begin{align*}
I_{\dom}(t)=|\bb|t\pthesis{\frac{1}{\pi}\Psi(t)-R_{\dom}(t)},
\end{align*}
where
\begin{align}\label{impfunct2}
R_{\dom}(t)&=\int_{0}^{\lo\,t^{-1}}dr r^d p_1(re_d) \gamma_{\dom}(tr)\geq 0,
\end{align}
with $\gamma_{\dom}$ as defined in  \eqref{impfunct}.
\end{prop}

\begin{proof}  We start by noticing that with the aid of  spherical coordinates and the fact that both the Poisson heat kernel  and $g_B(z)$ are  radial functions (see Lemma \ref{scvf}), we can rewrite $I_{\dom}(t)$ as follows.
\begin{align*}
I_{\dom}(t)&=|\bb|t\int_{0}^{\lo\,t^{-1}}dr\, r^d\, p_1(r\,e_d)\pthesis{\frac{\cvf\pthesis{0}-\cvf\pthesis{tre_d}}{tr}}\\ \nonumber
&=|\bb|t\int_{0}^{\lo\,t^{-1}}dr\, r^d\, p_1(r\,e_d)\pthesis{\frac{\cvf\pthesis{0}-\cvf\pthesis{tre_d}}{tr}-w_{d-1}+w_{d-1}}
\\ \nonumber
&=|\bb|t\pthesis{w_{d-1}\kappa_d\Psi(t)-R_B(t) }.
\end{align*}
Now, the desired result follows from the last expression by using  the fact that
 \eqref{vaunitball} and \eqref{kappadef} imply that
\begin{align*}
\kappa_{d}\,w_{d-1}=\Gamma\pthesis{\frac{d+1}{2}}\pi^{-\frac{d+1}{2}}\cdot\frac{\pi^{\frac{d-1}{2}}}{\Gamma\pthesis{1+\frac{d-1}{2}}}=\frac{1}{\pi}.
\end{align*}
Finally, $R_{\dom}(t)\geq 0$ is deduced from Proposition \ref{prop1}  and this concludes the proof.
\end{proof}

Now, we proceed to study the small time behavior of the  functions  $\Psi(t)$ and $R_{\dom}(t)$ defined in the previous proposition.

\begin{prop} \label{propimp}Let $\Psi(t)$ be the function defined in \eqref{2impf} for $t>0$. Then, $$\Psi(t)=\ln\pthesis{\frac{1}{t}}+F(t),$$
where 
\begin{align}\label{Fdec}
F(t)=\ln\pthesis{\lo+\sqrt{4+t^2}}+\int_{0}^{\arcsinh(\lo\,t^{-1})}(\tanh^d(\theta)-1).
\end{align}
Furthermore,
\begin{align}\label{lim}
\lim\limits_{\tgo}F(t)=2\ln(2)+\int_{0}^{\infty}d\theta \pthesis{\tanh^d(\theta)-1}.
\end{align}
\end{prop}
\begin{proof}
The change of variable $r=\sinh(\theta)$ turns the function $\Psi(t)$ defined in \eqref{2impf} into
$$\int_{0}^{\arcsinh(\lo\,t^{-1})}d\theta \tanh^d(\theta),$$
that together with  the identity
\begin{align*}
\arcsinh(\lo\,t^{-1})&=\ln\pthesis{\frac{1}{t}}+\ln\pthesis{\lo +\sqrt{t^2+4}},
\end{align*}
produces the desired identity \eqref{Fdec} since 
\begin{align*}
\Psi(t)&=\int_{0}^{\arcsinh(\lo\,t^{-1})}d\theta \pthesis{1+(\tanh^d(\theta)-1)}\\
&= \arcsinh(\lo\,t^{-1})+\int_{0}^{\arcsinh(\lo\,t^{-1})}d\theta \pthesis{\tanh^d(\theta)-1}.
\end{align*}

On the other hand, the following equality
$$\tanh^d(\theta)=\pthesis{1-\frac{2}{e^{2\theta}+1}}^d=1+\mysum{j}{1}{d} \binom{d}{j}\frac{(-2)^{j}}{(e^{2\theta}+1)^j},$$
and the fact that $$\int_{0}^{\infty}\frac{d\theta}{(e^{2\theta}+1)^{j}}\leq \int_{0}^{\infty}d\theta\,e^{-2\theta\,j} =\frac{1}{2j},$$
show that
\begin{align*}
\int_{0}^{\infty}d\theta \abs{\tanh^d(\theta)-1}\leq \mysum{j}{1}{d}\binom{d}{j}\frac{2^j}{2j}.
\end{align*}
Thus, an application of the Lebesgue dominated converge theorem shows that $\eqref{lim}$ holds true.
\end{proof}

We now continue with the study of  $R_{\dom}(t)$ as $\tgo$.

\begin{prop}\label{prop3} Let $R_{\dom}(t)$ be the function defined in $\eqref{impfunct2}$ and consider $\gamma_{\dom}$ as defined in $\eqref{impfunct}$. 
Then,
\begin{align}
\lim\limits_{\tgo}R_{\dom}(t)=\kappa_{d}\int_{0}^{1}ds
s^{-1}\gamma_{\dom}(\lo s),
\end{align}
with $\kappa_d$  as defined in
\eqref{kappadef}.
\end{prop}
\begin{proof} We know from $\eqref{impfunct2}$ that 
$$R_{\dom}(t)=\int_{0}^{\lo\,t^{-1}}dr r^d p_1(re_d) \gamma_{\dom}(tr),$$
for $t>0$.
Next, the change of variable $r=\lo t^{-1}s$  turns $R_{\dom}(t)$ into 
$$\lo^{d+1}\,\kappa_{d}\int_{0}^{1}ds \frac{s^{d}\gamma_{\dom}(\lo s)}{(t^2+\lo^2\,s^2)^{\frac{d+1}{2}}}.$$
Now, the basic inequality $\lo^2\,s^2\leq t^2+\lo^2\,s^2$ shows that
$$0\leq R_{\dom}(t)\leq \kappa_{d}\int_{0}^{1}ds
s^{-1}\gamma_{\dom}(\lo s).$$ In particular, we conclude
\begin{align}\label{1}
\varlimsup_{\tgo}R_{\dom}(t)\leq \kappa_{d}\int_{0}^{1}ds
s^{-1}\gamma_{\dom}(\lo s).
\end{align}

On the other hand, since $\gamma_{\dom}(2s)\geq 0$ for $0<s\leq 1$ due to Proposition \ref{ineqbasic},  we obtain by Fatou's Lemma  that
\begin{align}\label{2}
\varliminf_{\tgo}R_{\dom}(t)\geq \lo^{d+1}\,\kappa_{d}\int_{0}^{1}ds \varliminf_{\tgo}\frac{s^{d}\gamma_{\dom}(\lo s)}{(t^2+\lo^2\,s^2)^{\frac{d+1}{2}}}=\kappa_{d}\int_{0}^{1}ds
s^{-1}\gamma_{\dom}(\lo s).
\end{align}
Hence, the desired result follows from \eqref{1} and \eqref{2}. 
\end{proof}

Now that we have all the necessary facts at hand, we proceed to give the proof of our main result.

{\bf Proof of Theorem \ref{mthm}:}
Because of the identity \eqref{id1} and Proposition \ref{prop22} , we have that
\begin{align*}
|\dom|-\Halp&= |\dom|\phi(t)+ I_{\dom}(t)\\ 
&=|\dom|\phi(t)+|\bb|t\pthesis{\frac{1}{\pi}\Psi(t)-R_{\dom}(t)}
\end{align*}
Hence, by Proposition \ref{propimp}, we arrive at
\begin{align*}
|\dom|-\Halp&= |\dom|\phi(t)+|\bb|t\pthesis{\frac{1}{\pi}\ln\pthesis{\frac{1}{t}}+\frac{1}{\pi}F(t)-R_{\dom}(t)}.
\end{align*}
The last identity yields
\begin{align*}
\frac{1}{t}\pthesis{|\dom|-\Halp-\frac{|\bb|}{\pi}\,t\,\ln\pthesis{\frac{1}{t}}}&= |\dom|\frac{\phi(t)}{t}+|\bb|\pthesis{\frac{F(t)}{\pi}-R_B(t)}
\end{align*}
so that by combining together the results given in Propositions \ref{prop1 }, \ref{propimp} and \ref{prop3}, we arrive at

\begin{align*}
\lim\limits_{\tgo}&\frac{1}{t}\pthesis{|\dom|-\Halp-\frac{|\bb|}{\pi}\,t\,\ln\pthesis{\frac{1}{t}}}=\\ &\frac{|\dom||\bb|\kappa_d}{2}+|\bb|\pthesis{\frac{1}{\pi}\set{2\ln(2)+\int_{0}^{\infty}d\theta \pthesis{\tanh^d(\theta)-1}}-\kappa_{d}\int_{0}^{1}ds
s^{-1}\gamma_{\dom}(\lo s)},
\end{align*}
where the last expression is equivalent to \eqref{implim} after some simplifications.
\qed

\section{Computation of the third term for the unit ball in the plane and space}\label{sec:compu}

The goal of this section is  to calculate explicitly the limit in Theorem \ref{mthm} for the unit ball in the plane and space. This can be done because  the term $$\int_{0}^{1}ds s^{-1}\gamma_{\dom}(\lo\,s)$$ can be computed explicitly for $d=2$ and $d=3$.  Indeed, the above integral can be computed for any dimension $d\geq 2$, but calculations become complicated as the dimension gets higher. It is worth mentioning that in this section, we will only present the values of the integrals since they can be obtained by employing an integral calculator.

\begin{prop} Let $B=B_1(0)$ be the unit ball in $\R^d$. Then,
\begin{align}
\lim\limits_{\tgo}\frac{1}{t}\pthesis{|B|-\mathbb{H}_{B}(t)-\frac{|\bb|}{\pi} t\ln\pthesis{\frac{1}{t}}}=\begin{cases}
    6\ln(2)-2       & \quad \text{if } d=2,\\
   4\ln(2) & \quad \text{if } d =3.\\
  \end{cases}
\end{align}
\end{prop}
\begin{proof}
 By using identity \eqref{ineq00} and Lemma \ref{impl}, we have that
 $$ \gamma_{B}(2s)=\begin{cases}
    2-\sqrt{1-s^2}+\frac{1}{s}\pthesis{\arcsin(\sqrt{1-s^2})-\frac{\pi}{2}}       & \quad \text{if } d=2,\\
   \frac{\pi}{3} s^2  & \quad \text{if } d =3,\\
   \end{cases}$$
which in turn implies by using an integral calculator that
\begin{align}\label{intvalue}
    \int_{0}^{1}dss^{-1}\gamma_{B}(2s)=\begin{cases}
    \frac{\pi-4\ln(2)}{2}      & \quad \text{if } d=2,\\
   \frac{\pi}{6}  & \quad \text{if } d =3.\\
   \end{cases}
   \end{align}

Therefore, by combining together the values given in the   table below together with \eqref{intvalue}, we deduce the desired result. 
  \begin{table}[htb]
\centering
\begin{tabular}{|l|l|l|l|l|}
\hline
& \multicolumn{4}{c|}{Values for the unit ball} \\
\cline{1-5}
 $d$& \bigstrut$A_d=|\bb|$& $w_d=|B|$ & $\kappa_d$ & $\int_{0}^{\infty}d\theta(\tanh^d(\theta)-1)$\\
\hline 
\multirow{1}{1cm}{$2$}\bigstrut& $2\pi$ & $\pi$ & $\frac{1}{2\pi}$ & $-1$\\   \cline{1-5}
\multirow{1}{1cm}{$3$} \bigstrut& $4\pi$ & $\frac{4\pi}{3}$ & $\frac{1}{\pi^{2}}$  & $-\ln(2)-\frac{1}{2}$\\ \hline 
\end{tabular}
\label{tabla:final}
\end{table}

\end{proof}

{\bf Conclusions:} Based on the result entailed in Theorem \ref{mthm}, we conjecture that the spectral heat content $Q_B(t)$ defined in \eqref{shc} behaves 
as
$$Q_B(t)=|B|-C_d^{1}(B)t\ln\pthesis{\frac{1}{t}}-C_d^{2}(B)t+o(t), \tgo$$
 for some positive constants $C_d^{1}(B)$ y $C_d^{2}(B)$ depending on the geometry of the unit ball $B$ (they are not expected to have a nice and closed form) where our limit stated in Theorem \ref{mthm} might provide either an upper or lower bound for the constant $C_d^{2}(B)$ as it is done for $C_d^{1}(B)$ in \cite{Acu, TPS}. 
 
 On the other hand, some of the techniques and geometry objects  employed throughout the paper might be extended for other domains $\Omega$ different from the unit ball by working with functions of bounded variation (see for example \cite{Evans, Galerne, Mor}). However, the main difficulty arises from the fact that the covariance function $g_{\Omega}(z)$ cannot be found explicitly in general even for nice sets as open bounded convex sets where a different approach is required. 
\\

{\bf Acknowledgements}: This  investigation has been supported by Universidad de Costa Rica, project 1563.
\\

\end{document}